\RequirePackage{fix-cm}
\documentclass[smallextended]{svjour3m}       
\smartqed  
\usepackage{amsmath,amstext,amssymb}
\usepackage{graphicx}
\usepackage{float}
\usepackage{yfonts}

\newcommand{\C}{\textswab{C}}

\begin{document}

\title{On central Fubini-like numbers and polynomials
}
\subtitle{}


\author{Hacène Belbachir         \and
        Yahia Djemmada 
}


\institute{H. Belbachir \at
              USTHB, Faculty of Mathematics, RECITS Laboratory, BP 32, El Alia, 16111, Bab Ezzouar, Algiers, Algeria. \\
              \email{hacenebelbachir@gmail.com, hbelbachir@usthb.dz}           
           \and
           Y. Djemmada \at
              USTHB, Faculty of Mathematics, RECITS Laboratory, BP 32, El Alia, 16111, Bab Ezzouar, Algiers, Algeria.\\
              \email{yahia.djem@gmail.com, ydjemmada@usthb.dz}
}

\date{}

\maketitle

\begin{abstract}
We introduce \textit{the central Fubini-like} numbers and polynomials using Rota approach. Several identities and properties are established as generating functions, recurrences, explicit formulas, parity, asymptotics and determinantal representation.
\end{abstract}

\section{Introduction}
We start by giving some definitions that will be used throughout this paper.

The falling factorial $x_n$ is defined by
$$x_n=x(x-1)\cdots(x-n+1),$$
and the central factorial $x^{[n]}$, see \cite{Jr}, is defined by,
$$x^{[n]}=x(x+n/2-1)(x+n/2-2)\cdots(x-n/2+1).$$

It is well-known that, for all non-negative integers $n$ and $k$ ($k \leq n$), Stirling numbers of the second kind are defined as the coefficients $S(n,k)$ in the expansion
\begin{equation}\label{Sexp}
x^n=\sum_{k=0}^{n}S(n,k)x_k.
\end{equation}

Riordan, in his book \cite{Rd}, shows that, for all non-negative integers $n$ and $k$ ($k \leq n$), the central factorial numbers of the second kind are the coefficients $T(n,k)$ in the expansion
\begin{equation}\label{Texp7}
x^n=\sum_{k=0}^{n}T(n,k)x^{[k]}.
\end{equation}

In combinatorics, the number of ways to partition a set of $n$ elements into $k$ nonempty subsets are counted by Stirling numbers $S(n,k)$, and the central factorial numbers $T(2n,2n-2k)$ count the number of ways to place $k$ rooks on a 3D-triangle board of size $(n-1)$, see \cite{Kr}.
\begin{figure}[H]
\centering
\includegraphics[scale=0.7]{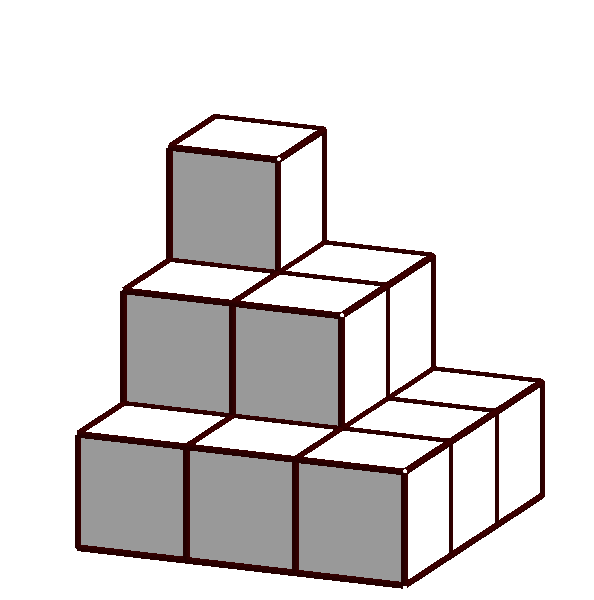}
\caption{3D-triangle board of size 3.}

\end{figure}

The coefficients $S(n,k)$ and $T(n,k)$ satisfy, respectively, the triangular recurrences
$$S(n,k)= kS(n-1 , k)+ S(n-1 , k-1) \text{\ \ \ } (k \leq n),$$
and
$$T(n,k)=\left(\frac{k}{2}\right)^2T(n-2,k)+T(n-2,k-2) \text{\ \ \ } (k \leq  n),$$
\\
where $\displaystyle{S(n,k)=T(n,k)=0}$ for $k > n$, and $\displaystyle{S(0,0)=T(0,0)=1} $.

$T(n,k)$ admits also the explicit expression 

\begin{equation}\label{Texp}
T(n,k)=\frac{1}{k!}\sum_{j=0}^{k} (-1)^j {k \choose j} \left( \frac{k}{2}-j\right)^n.
\end{equation}

The usual difference operator $\Delta$, the shift operator $\mathsf{E}^a$ and the central difference operator $\delta$ are given respectively by

$$\Delta f(x)=f(x+1)-f(x),$$

$$\mathsf{E}^a f(x)=f(x+a),$$
and
$$\delta f(x)=f(x+1/2)-f(x-1/2).$$

Riordan, \cite{Rd}, mentioned that the central factorial operator $\delta$ satisfy the following property
 
\begin{equation}
\delta f_n(x)=nf_{n-1}(x),
\end{equation}
where $f_n(x)$ is a sequence of polynomials satisfies $f_0(x)=1$.

The same operator $\delta$ can be expressed by the means of both $\Delta$ and $\mathsf{E}^a$, see \cite{Jr,Rd}, as follows

\begin{equation}\label{Rel1}
\delta f(x)=\Delta\mathsf{E}^{-1/2}f(x).
\end{equation}

For more details about difference operators, we refer the reader to \cite{Jr}.
\section{Central Fubini-like numbers and polynomials}

In 1975, Tanny \cite{Tn} introduced  the Fubini polynomials (or ordered Bell polynomials) $F_n(x)$ by applying a linear transformation $\mathcal{L}$ defined as,

$$\mathcal{L}(x_n):=n!x^n,$$
and define $F_n(x)$ by,
\begin{equation}\label{FnxDef}
F_n(x):=\mathcal{L}(x^n)=\mathcal{L}\left(\sum_{k=0}^{n}S(n,k)x_k\right)\sum_{k=0}^{n}S(n,k)\mathcal{L}(x_k)=\sum_{k=0}^{n}k!S(n,k)x^k,
\end{equation}
putting $x=1$ in \eqref{FnxDef}, we get
\begin{equation}\label{FnDef}F_n:=F_n(1)=\sum_{k=0}^{n}k!S(n,k),
\end{equation}
which is the $nth$ Fubini number.

Fubini polynomials $F_n(x)$ has the exponential generating function below, see \cite{Tn},
\begin{equation}
\sum_{n \geq 0}F_n(x)\frac{t^n}{n!}=\frac{1}{1-x(e^{t}-1)}.
\end{equation}

Recently, several authors have studied Fubini numbers and polynomials and giving some applications for these polynomials as in \cite{Jn,Kim1,Kim2,Kim4}.

For more details concerning Fubini numbers and polynomials, see \cite{Bo,Di,Gr,Me,Tn,Ze} and papers cited therein.

Now we introduce the linear transformation $\mathcal{Z}$ as,
\begin{definition}
For $n \geq 0$, we define the transformation
\begin{equation}
\mathcal{Z}(x^{[n]})=n!x^n.
\end{equation}
\end{definition}
then we have 
\begin{equation}
\mathcal{Z}(x^n)=\mathcal{Z}\left(\sum_{k=0}^{n}T(n,k)x^{[k]}\right)=\sum_{k=0}^{n}T(n,k)\mathcal{Z}(x^{[k]})=\sum_{k=0}^{n}k!T(n,k)x^k.
\end{equation}
And due to Formula \eqref{Texp}, we are now able to introduce the main tool of the paper,
\begin{definition}
The  $nth$ \textit{central Fubini-like} polynomial is given by,
\begin{equation}
\C_n(x):=\sum_{k=0}^{n}k!T(n,k)x^k.
\end{equation}
Setting $x=1$, we obtain the \textit{central Fubini-like numbers},  
\begin{equation}
\C_n=\C_n(1):=\sum_{k=0}^{n}k!T(n,k).
\end{equation}
\end{definition}
The first central polynomials $\C_n$ are,
\begin{table}[H]
$$
\begin{array}{ll}

\end{array}$$
$$\begin{array}{|l|l|l|}
\hline
n&\C_{2n}(x)&2^{2n}\C_{2n+1}(x)\\
\hline
0& 1 & x \\
\hline
1& 2 x^2 & x+24 x^3 \\
\hline
2& 2 x^2+24 x^4 & x+240 x^3+1920 x^5 \\
\hline
3& 2 x^2+120 x^4+720 x^6 & x+2184 x^3+67200 x^5+322560 x^7 \\
\hline
4&2 x^2+504 x^4+10080 x^6+ 40320 x^8 & x+19680 x^3+1854720 x^5+27095040 x^7+92897280 x^9 \\
\hline
\end{array}$$
\caption{First value of $\C_n(x)$.}
\end{table}
and the first $\C_n$ numbers are,

\begin{table}[h!]
$$\begin{array}{ll}
(\C_{2n})_{n \geq 0}:&1, 2, 26, 842, 50906, 4946282, 704888186, 138502957322, 35887046307866,\dots\\
\\
(2^{2n}\C_{2n+1})_{n \geq 0}:&1, 25, 2161, 391945, 121866721, 57890223865, 38999338931281, 
35367467110007785,\dots\\
\end{array}$$
\end{table}


\subsection{Exponential generating function}
We begin by establishing the exponential generating function of central Fubini-like polynomials,
\begin{theorem}\label{PEGF}
The polynomials $\C_n(x)$ have the following exponential generating function
\begin{equation}
G(x;t):=\sum_{n \geq 0}^{}\C_n(x)\frac{t^n}{n!}=\frac{1}{1-2x\sinh(t/2)}.
\end{equation}
\end{theorem}
\begin{proof}
We have
\begin{equation*}
\begin{split}
\sum_{n \geq 0}^{}\C_n(x)\frac{t^n}{n!}&=\sum_{0\leq k \leq n < \infty}^{}k!T(n,k)x^k\frac{t^n}{n!}=\sum_{k \geq 0}k!x^k\sum_{n =k }^{\infty}T(n,k)\frac{t^n}{n!},\\
\text{from \cite{Rd} p. 214, we have }&\\
\sum_{n \geq 0}^{}T(n,k)\frac{t^n}{n!}&=\frac{1}{k!}\left(2 \sinh (t/2)\right)^k,\\
\text{therefore \hspace{2.55cm}}&\\
\sum_{n \geq 0}^{}\C_n(x)\frac{t^n}{n!}&=\sum_{k = 0}^{\infty}  \left( 2 \sinh (t/2)\right)^kx^k=\frac{1}{1-2x\sinh(t/2)}.
\end{split}
\end{equation*}
\end{proof}
\begin{corollary}\label{NEGF}
The numbers $\C_n$  have the following exponential generating function
\begin{equation}
\sum_{k \geq 0}^{n}\C_n\frac{t^n}{n!}=\frac{1}{1-2\sinh(t/2)}.
\end{equation}
\end{corollary}
\subsection{Explicit representations}
In this subsection we propose some explicit formulas for the central Fubini-like polynomials, we start by the derivative representation.
\begin{proposition}\label{Prop1} The polynomials $\C_n(x)$ correspond to the higher derivative expression bellow,
$$\C_{n}(x)=\left.{\sum_{k\geq 0}^{\infty}\frac{d^n}{dt^n}\left(2x\sinh{(t/2)}\right)^{k}}\right\vert_{t=0}.$$
\end{proposition}
\begin{proof}
Let
\begin{equation*}
\begin{split}
\left.{\frac{d^n}{dt^n}\left(\sum_{m\geq 0}^{\infty}\C_m(x)\frac{t^m}{m!}\right)}\right\vert_{t=0}&=\left.{\sum_{m\geq n}^{\infty}\C_m(x)\frac{t^{m-n}}{(m-n)!}}\right\vert_{t=0}=\left.{\sum_{m\geq 0}^{\infty}\C_{n+m}(x)\frac{t^{m}}{m!}}\right\vert_{t=0}=\C_{n}(x).\\
\end{split}
\end{equation*}
Thus from Theorem \ref{PEGF} we get the result.
\end{proof}
\begin{remark}
From Formula \eqref{Texp}, it is clear that the central Fubini-like polynomials satisfy the following explicit formula
$$
\C_n(x)=\left.{\sum_{k=0}^{\infty}x^k\sum_{j=0}^{k}(-1)^j {k \choose j}(k/2-j)^n}\right.
.$$
\end{remark}
\begin{proof}
Let
\begin{equation*}
\begin{split}
\left.{\sum_{k=0}^{\infty}\frac{d^n}{dt^n}\left(2x\sinh{(t/2)}\right)^{k}}\right\vert_{t=0}&=\left.{\sum_{k=0}^{\infty}\frac{d^n}{dt^n}\left(xe^{-t/2}(e^t-1)\right)^{k}}\right\vert_{t=0}\\
&=\left.{\sum_{k=0}^{\infty}\frac{d^n}{dt^n}x^ke^{-kt/2}\sum_{j=0}^{k}(-1)^j {k \choose j}e^{(k-j)t}}\right\vert_{t=0}\\
&=\left.{\sum_{k=0}^{\infty}x^k\sum_{j=0}^{k}(-1)^j {k \choose j}\frac{d^n}{dt^n}e^{(k/2-j)t}}\right\vert_{t=0}\\
&=\left.{\sum_{k=0}^{\infty}x^k\sum_{j=0}^{k}(-1)^j {k \choose j}(k/2-j)^n}\right..
\\
\end{split}
\end{equation*}
\end{proof}
\begin{theorem}\label{Thm2}
For non-negative $n$, The following explicit representation holds true,
\begin{equation}\label{Comp}
\C_n(x)=x\sum_{k=0}^{n-1}{n \choose k}\sum_{j=0}^{k}{k \choose j}\Big(\frac{-1}{2}\Big)^{k-j}\C_j(x)=x\sum_{j=0}^{n-1}{n \choose j}\delta [0^{n-j}]\C_j(x).
\end{equation}
where $\delta[0^{n-j}]=(1/2)^{n-j}-(-1/2)^{n-j}$.
\end{theorem}
The proof will depend on Lemma \ref{lm1}, Lemma \ref{lm2} and Relation \eqref{Rel1}.
\begin{lemma}\label{lm1}For all polynomials $p_n(x)$ the following relation holds true
$$\mathcal{Z}(p_n(x))=x\mathcal{Z}(\delta p_n(x)).$$
\end{lemma}
\begin{proof} We have
\begin{equation*}
\begin{split}
\mathcal{Z}(x^{[n]})&=n!x^n=xn(n-1)!x^{n-1}=x\mathcal{Z}(nx^{[n-1]})=x\mathcal{Z}(\delta x^{[n]}),
\end{split}
\end{equation*}
as any polynomial can be writing as sum of central factorials ($x^{[n]}$) we have our result.
\end{proof}
\begin{lemma}[Tanny \cite{Tn}]\label{lm2}
For all polynomials $p_n(x)$ we have
\begin{equation}
\Delta p_n(x)=\sum_{k=0}^{n-1}{n \choose k} p_k(x).
\end{equation}
\end{lemma}
Now we give the proof of Theorem \ref{Thm2},
\begin{proof}
Using previous lemmas and setting $p_n(x)=x^n$, we get  
\begin{equation*}
\begin{split}
\mathcal{Z}(x^n)&=x\mathcal{Z}(\delta x^n)=x\mathcal{Z}\Big(\Delta E^{-1/2}x^n\Big)=x\mathcal{Z}\Big(\Delta \Big(x-\frac{1}{2}\Big)^n\Big)\\&=x\mathcal{Z}\Big(\sum_{k=0}^{n-1}{n \choose k} \Big(x-\frac{1}{2}\Big)^k\Big)=x\mathcal{Z}\Big(\sum_{k=0}^{n-1}{n \choose k} \sum_{j=0}^{k}{k \choose j} \Big(\frac{-1}{2}\Big)^{k-j}x^j\Big)\\
&=x\sum_{k=0}^{n-1}{n \choose k}\sum_{j=0}^{k}{k \choose j}\Big(\frac{-1}{2}\Big)^{k-j}\C_j(x),
\end{split}
\end{equation*}
using binomial product identity $\displaystyle {n \choose k}{ k \choose j}={n-j\choose k-j}{n \choose j}$ we get the result.
\end{proof}
\begin{corollary}\label{cor2}
The central Fubini-like numbers satisfy
\begin{equation}
\C_n=\sum_{j=0}^{n-1}{n \choose j}\delta [0^{n-j}]\C_j.
\end{equation}
\end{corollary} 
\begin{proof}
Setting $x=1$ we have the result.
\end{proof}
Now we give an explicit formula connecting the $n$th central Fubini-like polynomials with Stirling numbers of the second kind $S(n,k)$,
\begin{theorem}  \label{Strconv} The central Fubini-like polynomials $\C_n(x)$ satisfy,
\begin{equation}
\C_n(x)=\sum_{k=0}^{\infty}k!x^k\sum_{j=0}^{n} {n \choose j} \left( \frac{-k}{2} \right)^j S(n-j,k).
\end{equation}
\end{theorem}
\begin{proof}
From Theorem \ref{PEGF}, we have
\begin{equation*}
\sum_{n \geq 0}\C_n(x)\frac{t^n}{n}=\frac{1}{1-2x \sinh(t/2)},
\end{equation*} 
using the exponential form of $2x\sinh{(t/2)}$ we get,
\begin{equation*}
\begin{split}
\sum_{n \geq 0}\C_n(x)\frac{t^n}{n}&=\frac{1}{1-xe^{(-t/2)}(e^t-1)}=\sum_{k\geq 0}^{}x^ke^{(-kt/2)}(e^t-1)^k\\
\text{But it is known that,\ \ }&\\
&\sum_{n \geq 0}S(n,k)\frac{t^n}{n!}=\frac{(e^t-1)^k}{k!}, \\
\text{therefore \quad \quad \quad \quad \quad}&\\
\sum_{n \geq 0}\C_n(x)\frac{t^n}{n}&=\sum_{k\geq 0}^{}x^kk!\sum_{n\geq 0}{\left(\frac{-k}{2}\right)^n\frac{t^n}{n!}}\sum_{n \geq 0}S(n,k)\frac{t^n}{n!}.
\end{split}
\end{equation*} 
Then Cauchy's product implies the identity.
\end{proof} 
\begin{corollary}
The central Fubini-like numbers $\C_n$ satisfy,
\begin{equation}
\C_n=\sum_{k=0}^{\infty}k!\sum_{j=0}^{n} {n \choose j} \left( \frac{-k}{2} \right)^j S(n-j,k).
\end{equation}
\end{corollary}
\subsection{Umbral representation}
In the following theorem we  use the  umbral notation $\C_k \equiv \C^k$, 
\begin{theorem}\label{Thm5}
Let $n$ be a non-negative integer, for all real $x$ we have 
\begin{equation*}
\C_n(x)=x\left[(\C(x)+1/2)^n-(\C(x)-1/2)^n\right].
\end{equation*}
\end{theorem}
\begin{proof}
From Theorem \ref{Thm2} and using the umbral notation, a simple calculation gives the umbral representation result.
\end{proof}
\begin{corollary} For non-negative integer $n$, we have 
\begin{equation*} 
\C_n=(\C+1/2)^n-(\C-1/2)^n.
\end{equation*}
\end{corollary}
\begin{proof}
It suffices to set $x=1$ in Theorem \ref{Thm5}.
\end{proof}
\subsection{Parity}
\hfill\\
A function $f(x)$ is said to be even when  $f(x) = f(-x)$ for all $x$ and it is said to be odd when  $f(x) = -f(-x)$

\begin{theorem}\label{Thm3}
For all non-negative $n$ and real variable $x$ we have
\begin{equation*}
\C_n(x)=(-1)^n\C_n(-x).
\end{equation*}
\end{theorem}
\begin{proof}
It is known that the function $\sinh(t)$ is odd, this gives
$G(x;t)=G(-x;-t),$
then comparing the coefficients of $t^n/n!$ in $G(x;t)$ and $G(-x;-t)$ the theorem follows.
\end{proof}
\begin{corollary}\label{Thm4}
The polynomials $\C_n(x)$ are odd if and only if $n$ is odd and even else.
\end{corollary}
\begin{proof}
Using Theorem \ref{Thm3}, it suffices to replace $n$ by $2k+1$ (and $2k$) and establish the property.
\end{proof}
\subsection{Recurrences and derivatives of higher order}
Now we are interested to derive some recurrences for $\C_n(x)$ in terms of their derivatives.

First, we deal with a recurrence of second order, 
\begin{theorem}\label{Thm6}
For $n\geq 2$, the polynomials $\C_n(x)$ satisfy the following recurrence,
$$\C_{n}(x)=2x^2\C_{n-2}(x)+ \left(\frac{x}{4}+4x^3\right)\C'_{n-2}(x)+\left(\frac{x^2}{4}+x^4\right)\C''_{n-2}(x).$$
\end{theorem}
\begin{proof}
We use the recurrence $\displaystyle T(n,k)=T(n-2,k-2)+\frac{k^2}{4}T(n-2,k)$, in fact 
\begin{equation*}
\begin{split}
 \C_n(x)&=\sum_{k=0}^{n}k!T(n,k)x^k  \\
&  =\sum_{k=2}^{n}k!T(n-2,k-2)x^k+\frac{1}{4}\sum_{k=0}^{n}k^2k!T(n-2,k)x^k  \\
&  =\sum_{k=0}^{n}(k+2)!T(n-2,k)x^{k+2}+\frac{x}{4}  \left( \sum_{k=0}^{n}kk!T(n-2,k)x^k\right)' \\
&=x^2 \left( x^2\sum_{k=0}^{n}k!T(n-2,k)x^{k} \right)''+\frac{x}{4} \left(x\left(\sum_{k=0}^{n}k!T(n-2,k)x^k\right)'\right)'\\
& =x^2\left(x^2\C_{n-2}(x)\right)''+\frac{x}{4}\left(x\C'_{n-2}(x)\right)' \\
&=2x^2\C_{n-2}(x)+ \left(\frac{x}{4}+4x^3\right)\C'_{n-2}(x)+\left(\frac{x^2}{4}+x^4\right)\C''_{n-2}(x),
\end{split}
\end{equation*}

this concludes the proof.
\end{proof}

In the next Theorem we give a recurrence formula for the $rth$ derivative of $\C_n(x)$ 
\begin{proposition}\label{Pro2}
The $r^{th}$ derivative of $G(x;t)$ is given by,
$$\frac{\mathrm{d}^r}{\mathrm{d}^rx} G(x;t)=\frac{r!}{x^r}G(x;t)(G(x;t)-1)^r.$$
\end{proposition}
\begin{proof}
Induction on $r$ implies the equality.
\end{proof}
\begin{theorem}\label{Thm8}
Let $\C_n^{(r)}(x)$ be the $rth$ derivative of $\C_n(x)$. Then $\C_n^{(r)}(x)$ is given by,
{\small $$\C_n^{(r)}(x)=\frac{r!}{x^{r}}\sum_{k=0}^{r}{r \choose k}(-1)^{r-k}\sum_{{\tiny j_0+j_1+\cdots+j_{k}=n}}{n \choose j_0,j_1,\dots,j_{{\tiny k}}} \C_{j_0}(x)\C_{j_1}(x) \cdots \C_{j_{k}}(x).$$}
\end{theorem}
\begin{proof}
Using Proposition \ref{Pro2}, by applying Cauchy product and comparing the coefficients of $t^n/n!$, we get the result.
\end{proof}
\begin{corollary}\label{Thm7}
The following equality holds for any real $x$,
$$x\C_n'(x)=\sum_{k=0}^{n-1}{n \choose k} \C_k(x)\C_{n-k}(x).$$
\end{corollary}
\begin{proof}
Setting $r=1$ in Proposition \ref{Pro2} gives the first derivative of $G(x;t)$ as,
\begin{equation*}
\begin{split}
\frac{\mathrm{d}}{\mathrm{dx}}G(x;t)&=\frac{2\sinh\left(\frac{t}{2}\right)}{\left(1-2x\sinh(\frac{t}{2})\right)^2}=\frac{G(x;t)}{x}\left(G(x;t)-1\right),\\
x\frac{\mathrm{d}}{\mathrm{dx}}G(x;t)&=G(x;t)^2-G(x;t),\\
x\sum_{n \geq 0}\sum_{k=0}^{n}\C'_n(x)\frac{t^n}{n!}&=\left(\sum_{n \geq 0}\sum_{k=0}^{n}\C_n(x)\frac{t^n}{n!}\right)^2-\sum_{n \geq 0}\sum_{k=0}^{n}\C_n(x)\frac{t^n}{n!},
\end{split}
\end{equation*}
then applying the Cauchy product in the right hand side and comparing the coefficients of $t^n/n!$ we get the result.
\end{proof}
\subsection{Integral representation}
Integral representation is a fundamental property in analytic combinatorics.
The central Fubini-like polynomials can be represented by,
\begin{theorem}\label{Thm9}
The polynomials $\C_{n}(x)$ satisfy,
$$\C_{n}(x)=\frac{2 n!}{\pi } \mathbf{Im} \int_0^{\pi } \frac{\sin (n \theta)}{1-2x \sinh \left(e^{i \theta }/2\right)} \,  \mathrm{d}\theta. $$
\end{theorem}
\begin{proof}
We will use here the known identity, see \cite{Ca},
$$k^n=\frac{2n!}{\pi}\mathbf{Im} \int_0^{\pi}\exp{({ke^{i \theta}})}\sin(n\theta)d\theta.$$
we have,
\begin{equation*}
\begin{split}
\C_n(x)&=\sum_{k = 0}^{\infty}k!T(n,k)x^k\\&=\sum_{k = 0}^{\infty}x^k\sum_{j = 0}^{k}(-1)^j{k \choose j} \left( \frac{k}{2}-j \right)^n\\
&=\sum_{k = 0}^{\infty}x^k\sum_{j = 0}^{k}(-1)^j{k \choose j}\frac{2n!}{\pi}\mathbf{Im} \int_0^{\pi}\exp{\left((k/2-j)e^{i \theta}\right)}\sin(n\theta)\mathrm{d}\theta\\
&=\frac{2n!}{\pi}\mathbf{Im} \int_0^{\pi} \sin(n\theta)\sum_{k = 0}^{\infty}x^k \exp{\left(-\frac{k}{2}e^{i\theta}\right)}\left(\exp{(e^{i\theta})} - 1\right)^k\mathrm{d}\theta\\
&=\frac{2n!}{\pi}\mathbf{Im} \int_0^{\pi}\frac{\sin(n\theta)}{1-2 x \sinh \left(e^{i\theta }/2\right)}\mathrm{d}\theta.
\end{split}
\end{equation*}
\end{proof}

\subsection{Determinantal representation}
A determinantal representation of a sequence $(a_n)_{n\geq 0}$ is the expression of this one as a determinant of a given matrix $M$. Several papers have been published on determinantal representations of many sequences as Bernoulli numbers, Euler numbers, ordered Bell numbers (or Fubini numbers),...

Komatsu and Ram\'irez in a recent paper gives the following Theorem,
\begin{theorem}[Komatsu \& Ram\'irez \cite{Ku}]

Let $\{ \alpha_n\}_{n\geq 0}$ be a sequence with $\alpha_0=1$, and $R(j)$ be a function independent of $n$. Then
\begin{equation} \label{Matrix}
\alpha_n = \begin{vmatrix}
R(1)&1&&&\\
R(2)&R(1)&&&\\
\vdots&\vdots&\ddots&1&\\
R(n-1)&R(n-2)&\cdots&R(1)&1\\
R(n)&R(n-1)&\cdots&R(2)&R(1)\\
\end{vmatrix}.
\end{equation}
if and only if 
\begin{equation}\label{CompF}
\alpha_n=\sum_{j=1}^{n}(-1)^{j-1}R(j)\alpha_{n-j} \qquad (n \geq 1)
\end{equation}
with $\alpha_0$.
\end{theorem}
By applying the previous Theorem we get,
\begin{theorem}
For $n \geq 1$, we have
\begin{equation}
\frac{\C_n(x)}{n!}= \begin{vmatrix}
R(1)&1&&&\\
R(2)&R(1)&&&\\
\vdots&\vdots&\ddots&1&\\
R(n-1)&R(n-2)&\cdots&R(1)&1\\
R(n)&R(n-1)&\cdots&R(2)&R(1)\\
\end{vmatrix}.
\end{equation}
where $$R(j)=x\frac{(-1)^{j-1}}{j!}\delta [0^{j}]=x\frac{(-1)^{j-1}}{j!}\left(\left(\frac{1}{2}\right)^j-\left(-\frac{1}{2}\right)^j\right).$$
\end{theorem}
\begin{proof}
From Theorem \ref{Thm2} we have,
\begin{equation*}
\begin{split}
\C_n(x)&=x\sum_{j=0}^{n-1}{n \choose j}\delta [0^{n-j}]\C_j(x)=x\sum_{j=1}^{n}{n \choose j}\delta [0^{j}]\C_{n-j}(x)\\
\frac{\C_n(x)}{n!}&=\sum_{j=1}^{n}\frac{x}{j!}\delta [0^{j}]\frac{\C_{n-j}(x)}{(n-j)!}\\
\end{split}
\end{equation*}
It suffices to set $\alpha_n=\frac{\C_n(x)}{n!}$ and $R(j)=x\frac{(-1)^{j-1}}{j!}\delta [0^{j}]$
to get the result.
\end{proof}
\begin{remark}\label{RemR}
The function $R(j)=0$ if $j$ is even.
\end{remark}
Using the Remark $\ref{RemR}$, we can give a new binomial convolution for  $\C_n(x)$ polynomials,
\begin{theorem}For $n \geq 0$ we have,
\begin{equation}\label{Compp} 
\C_{n+1}(x)=x\sum_{k=0}^{\lfloor\frac{n}{2}\rfloor}4^{-k}{n+1 \choose 2k+1}\C_{n-2k}(x),
\end{equation}
\end{theorem}
\begin{proof}
From Remark \ref{RemR} and using Formula \eqref{CompF} with $\alpha_n=\C_n(x)/n!$ and $R(j)=x\frac{(-1)^{j-1}}{j!}\left(\left(\frac{1}{2}\right)^j-\left(-\frac{1}{2}\right)^j\right)$ we can establish the result.
\end{proof}
\begin{remark}
Formula \eqref{Compp} is better than result of Theorem \eqref{Thm2} from a computational point of view.
\end{remark}
\subsection{Asymptotic result with respect to $\C_n$}

Find an asymptotic behaviour a sequence $(a_n)_{n\geq 0}$ means to find a second function of $n$ simple than $a_n$ which gives a good approximation to the values of $a_n$ when $n$ is large.

In this subsection, we are interested to obtain the asymptotic behaviour of the central  Fubini-like numbers.

Let $(a_n)_{n \geq 0} $ be a sequence of non-negative real numbers, the asymptotic behaviour $a_n$ is closely tied to the poles in $G(z)$, where $G(z)$ is the generating function of $a_n$,
$$G(z)=\sum_{n \geq 0}a_n z^n.$$

Wilf, in his book \cite{Wi}, gives a method to determine the asymptotic behaviour $a_n$ which can be summarized in the following steps,
\begin{itemize}
\item[-] Find the poles $z_0,z_1,\dots,z_s$ in $G(z)$ .
\item[-] Calculate the principal parts $P(G(z),z_i)$ at each pole $z_i$ as bellow,
$$P(G(z),z_i)=\frac{Res(G(z),z_i)}{(z-z_i)},$$
where $Res(G(z),z_i)$ is the residue of $G(z)$ at the pole $z_i$.
\item[-]Set $H(z)=\sum_{i=0}^{s}P(G(z),z_i)$ then write $H(z)$ as the expansion bellow,
$$H(z)=\sum_{n \geq 0}b_n z^n,$$
\item[-]The sequence $(b_n)_{n \geq 0}$ is the asymptotic behavior $a_n$ when $n$ is big enough,
$$a_n \sim b_n + O\left(\left(\frac{1}{R}\right)^n\right), \quad n \longmapsto \infty.$$ 
where $R$ the smallest modulus of the poles.
\end{itemize} 
\begin{remark}
Poles $z_0,z_1,\dots,z_s$ are considered as simple poles (has a multiplicity equal to $1$).
\end{remark}

Analytic methods of determining the asymptotic behavior of a sequence $a_n$ are widely discussed on \cite{Be,Od,PR,Wi}.
\begin{theorem} Asymptotic behaviour of the $\C_n$ is given by  
$$\C_n \sim  \frac{ n!}{2^{n}\sqrt{5}} \mathfrak{Re}\left(\log ^{-n-1}(\phi )-(i \pi +\log(\phi -1
   ))^{-n-1}\right)+O\left(n! (0.053)^n\right), \quad n \longmapsto \infty $$
where $\phi$ is the \textit{Golden ratio}.   
\end{theorem}
\begin{proof}
Applying the previous steps in the generating function $G(t)=\frac{1}{1-2\sinh(t/2)}$ gives the result.
\end{proof}

\end{document}